\documentclass[11pt, notitlepage]{article}
\usepackage{amssymb,amsmath,comment}
\catcode`\@=11 \@addtoreset{equation}{section}
\def\thesection{\arabic{section}}

\def\theequation{\thesection.\arabic{equation}}
\catcode`\@=12
\usepackage{colortbl}%
\usepackage[mathscr]{eucal}
\usepackage{epsf}
\usepackage{a4wide}
\def\R{\mathbb{R}}

\newcommand{\e}{\epsilon}

\newcommand{\ga} {\gamma}

\newcommand{\Om} {\Omega}
\newcommand{\ra} {\rightarrow}

\newcommand{\De} {\Delta}
\newcommand{\la} {\lambda}
\newcommand{\La} {\Lambda}
\newcommand{\noi} {\noindent}

\newcommand{\uline} {\underline}
\newcommand{\oline} {\overline}
\newcommand{\mb} {\mathbb}
\newcommand{\mc} {\mathcal}

\usepackage[all]{xy}
\catcode`\@=11
\def\theequation{\@arabic{\c@section}.\@arabic{\c@equation}}
\catcode`\@=12

\newtheorem{Theorem}{Theorem}[section]
\newtheorem{Lemma}[Theorem]{Lemma}
\newtheorem{Proposition}[Theorem]{Proposition}

\newtheorem{Remark}[Theorem]{Remark}
\newtheorem{Definition}[Theorem]{Definition}

\begin{document}

{\vspace{0.01in}}

\title
{ \sc A Global multiplicity result for a very singular critical nonlocal equation}

\author{J. Giacomoni\footnote{Universit\'e de Pau et des Pays de l'Adour-E2S, CNRS, LMAP (UMR 5142) Bat. IPRA,
  Avenue de l'Universit\'e,
   64013 Pau cedex, France. e-mail:jacques.giacomoni@univ-pau.fr}, ~~ T. Mukherjee\footnote{Department of Mathematics, Indian Institute of Technology Delhi,
Hauz Khaz, New Delhi-110016, India.
 e-mail: tulimukh@gmail.com}~ and ~K. Sreenadh\footnote{Department of Mathematics, Indian Institute of Technology Delhi,
Hauz Khaz, New Delhi-110016, India.
 e-mail: sreenadh@gmail.com} }

\date{}

\maketitle

\begin{abstract}

\noi In this article, we show the global multiplicity result for the following nonlocal singular problem
\begin{equation*}
 (P_\la):\;\quad (-\De)^s u =  u^{-q} + \la u^{{2^*_s}-1}, \quad u>0 \; \text{in}\;
\Om,\quad u = 0 \; \mbox{in}\; \mb R^n \setminus\Om,
\end{equation*}
where  $\Om$ is a bounded domain in $\mb{R}^n$ with smooth boundary $\partial \Om$, $n > 2s,\; s \in (0,1),\; \la >0,\; q>0$ satisfies $q(2s-1)<(2s+1)$ and $2^*_s=\frac{2n}{n-2s}$. Employing the variational method, we show the existence of at least two distinct weak positive solutions for $(P_\la)$ in $X_0$ when $\la \in (0,\La)$ and no solution when $\la>\La$, where $\La>0$ is appropriately chosen. We also prove a result of independent interest that any weak solution to $(P_\lambda)$ is in $C^\alpha(\R^n)$ with $\alpha=\alpha(s,q)\in (0,1)$. The asymptotic behaviour of weak solutions reveals that this result is sharp.
\medskip

\noi \textbf{Key words:} Fractional Laplacian, very singular nonlinearity, variational method, H\"older regularity.

\medskip

\noi \textit{2010 Mathematics Subject Classification:} 35R11, 35R09, 35A15.

\end{abstract}

\section{Introduction}
In this article, we prove the existence, multiplicity and H\"older regularity of weak {solutions} to the following fractional critical and singular elliptic equation
\begin{equation*}
 (P_\la):\;\quad (-\De)^s u =  u^{-q} + \la u^{{2^*_s}-1}, \quad u>0 \; \text{in}\;
\Om,\quad u = 0 \; \mbox{in}\; \mb R^n \setminus\Om,
\end{equation*}
where  $\Om$ is a bounded domain in $\mb{R}^n$ with smooth boundary $\partial \Om$, $n > 2s,\; s \in (0,1),\; \la >0,\; q>0$ satisfies $q(2s-1)<(2s+1)$ and $2^*_s=\frac{2n}{n-2s}$. The fractional Laplace operator  denoted by $(-\De)^s$ is defined as
$$ (-\De)^s u(x) = 2C^n_s\mathrm{P.V.}\int_{\mb R^n} \frac{u(x)-u(y)}{\vert x-y\vert^{n+2s}}\,\mathrm{d}y$$
{where $\mathrm{P.V.}$ denotes the Cauchy principal value and $C^n_s=\pi^{-\frac{n}{2}}2^{2s-1}s\frac{\Gamma(\frac{n+2s}{2})}{\Gamma(1-s)}$, $\Gamma$ being the Gamma function.} The fractional power of Laplacian is the infinitesimal generator of {L\'evy} stable diffusion process and arise in anomalous diffusion in plasma, population dynamics, geophysical fluid dynamics, flames propagation, chemical reactions in liquids and American options in finance, see \cite{da} for instance. The theory of fractional Laplacian and elliptic equations involving it as the principal part has been evolved immensely in recent years. There is a vast literature available on it, however we cite \cite{buccur,BRS} for motivation to readers.

\noi The fractional elliptic equations with singular and critical {nonlinearities} was first studied by Barrios et al. in \cite{peral}. The authors considered the problem
\begin{equation*}
(-\De)^s u = \la\frac{f(x)}{u^\ga} + M u^{p}, \quad
 u>0\;\text{in}\;\Om, \quad
 u = 0 \; \mbox{in}\; \mb R^n \setminus\Om,
\end{equation*}
where  $n>2s$, $M\ge 0$, $0<s<1$, $\ga>0$, $\la>0$, $1<p<2_{s}^{*}-1$ and $f\in L^{m}(\Om)$, $m\geq 1$ is a nonnegative function. Here, authors studied the existence of distributional solutions using the uniform estimates of  $\{u_n\}$ which are solutions of the regularized problems with singular  term $u^{-\ga}$ replaced by $(u+\frac{1}{n})^{-\ga}$. Motivated by their results, Sreenadh and Mukherjee in \cite{TS-ejde} studied the singular problem
\begin{equation*}
 \quad (-\De)^s u = \la a(x)u^{-q} + u^{{2^*_s}-1}, \quad u>0 \; \text{in}\;
\Om,\quad u = 0 \; \mbox{in}\; \mb R^n \setminus\Om,
\end{equation*}
where $\la >0,\; 0 < q \leq 1 $ and $\theta \leq a(x) \in L^{\infty}(\Om)$, for some $\theta>0$. They showed that although the energy functional corresponding {to} this problem fails to be {Fr\'echet} differentiable, making use of its G\^{a}teaux differentiability the Nehari manifold technique can still be benefitted to obtain existence of at least two solutions over a certain range of $\la$. The significance of $q$ being less than $1$ is the G\^{a}teaux differentiability of the functional corresponding to the problem. Whereas if we look at the case $q>1$ then
the  functional
\[J(u)=\frac{C^n_s}{2}\|u\|_{H^{s}_{0}(\Om)}^{2}  -\frac{1}{1-q}\int_\Om |u|^{1-q}~dx-\frac{\la}{2^*_s}\int_\Om |u|^{2^*_s}~dx\]
may not be defined on the whole space nor it is even continuous on $D(I) \equiv \{u \in H^{s}_{0}(\Om) : I(u) < \infty\}$ {and this approach can not be extended}. Besides this, {one has that the} interior of $D(I)=\emptyset$ because of the singular term. But we notice that if we enforce the condition $q>1$ satisfies $q(2s-1)<(2s+1)$ then {we can prove that $D(I)$ is non empty and G\^ateaux
 differentiable on a suitable convex cone of $X_0$.}

 \noi The existence of weak solutions to $(P_\la)$ when $\la \in (0,\La)$ and no solution when $\la>\La$ has been already obtained by Giacomoni et al. in \cite{TJS-ANA}. But here the multiplicity of solutions has been achieved in $L^1_{loc}(\Om)$ only, {by using non smooth critical point theory, so the questions of existence of solutions in the energy space and  of H\"older regularity were still pending. This article is bringing answers to these two issues. For that, we} followed the approach of \cite{haitao} but we notify that the adversity and novelty of this article lies in extending Haitao's technique in a nonlocal framework.  The regularity of weak solution of the purely singular problem
\[(-\De)^s u = u^{-q},\; u>0,\;\text{in}\;\Om,\;\; u=0 \; \text{in}\mb R^n\setminus \Om\]
plays a vital role in our study. This has been obtained by Adimurthi, Giacomoni and Santra in \cite{AJS} in recent times. {In the present paper we extend the H\"older regularity results in our framework of weak solutions (see definition 1.1 below) rather than the more restricted classical solutions framework defined in \cite{AJS}. It requires additional $L^\infty$-estimates and the use of the weak comparison principle.} Nowadays, researchers are inspecting on various forms of singular nonlocal equations. We cite \cite{FP,cai-chu,chen} as some contemporary woks related to it.

\noi Our paper has been organized as follows- Section $2$ contains the function space setting along with some preliminary results. Section $3$ and $4$ contains the proof of existence of first and second weak solution to $(P_\la)$ respectively {(Theorem \ref{maintheorem-gm})}. {The proof of the h\"older regularity result (Theorem \ref{gm-reg2}) is done in Section $4$ based on a priori estimates proved in the Appendix}.

\begin{Definition}
A function $u \in X_0$ is said to be a weak solution of $(P_\la)$ if there exists a $m_K>0$ such that $ u>m_K$ in every compact subset $K$ of $ \Om$,  and it satisfies
\[\int_Q\frac{(u(x)-u(y))(\phi(x)-\phi(y))}{|x-y|^{n+2s}}~dxdy= \int_\Om (u^{-q}+u^{2^*_s-1})\phi~dx,\; \text{for all}\; \phi \in X_0.\]
\end{Definition}
{
Let $\phi_{1,s}$ be the first positive normalized eigenfunction ($\Vert\phi_{1,s}\Vert_{L^\infty(\Omega)}=1$) of $(-\Delta)^s$ in $X_0$.
We recall that $\phi_{1,s}\in C^s(\R^N)$ and $\phi_{1,s}\in C_{\delta^s}^+(\Omega)$ where $\delta(x)= \text{dist}({x,\partial \Om})$ (see for instance Proposition 1.1 and Theorem 1.2 in \cite{Ros-oton-serra-JMPA}). We then define  the function $\phi_q$ as follows:
\begin{eqnarray}\label{weighted-funct}
\phi_q=\displaystyle\left\{\begin{array}{lll}
& \phi_{1,s}\quad\mbox{ if } 0<q<1,\nonumber\\
& \phi_{1,s}\left(\ln\left(\frac{2}{\phi_{1,s}}\right)\right)^{\frac{1}{q+1}}\quad\mbox{ if } q=1,\nonumber\\
& \phi_{1,s}^{\frac{2}{q+1}}\quad\mbox{ if } q>1.
\end{array}\right.
\end{eqnarray}}
We prove the following as the main {results}:
\begin{Theorem}\label{maintheorem-gm}
There exists $\La>0$ such that
\begin{itemize}
 \item[(i) ] $(P_\la)$ admits at least two solutions in {$X_0\cap C^+_{\phi_q}(\Omega)$} for every $\la\in(0,\La)$.
 \item[(ii)]  $(P_\la)$ admits no solution for $\la>\La$.
 \item[(iii)] $(P_\La)$ admits at least one positive solution {$u_\La\in X_0\cap C^+_{\phi_q}(\Omega)$}.
\end{itemize}
 \end{Theorem}
\begin{Theorem}\label{gm-reg2}
Let $\la \in (0,\Lambda]$, $q>0$ satisfies $q(2s-1)<(2s+1)$ and $u \in X_0$ is any positive weak solution of $(P_\la)$ then
\begin{enumerate}
\item[(i)] $u \in C^s(\mb R^n)$ when $0<q<1$;
\item[(ii)] $u \in C^{s-\e}(\mb R^n)$ for any small enough $\e>0$ when $q=1$;
\item[(iii)] {$u \in  C^{\frac{2s}{q+1}}(\mb R^n)$ when $q>1$.}
\end{enumerate}
\end{Theorem}
{\begin{Remark}
Here, the  H\"{o}lder regularity for the weak solutions of $(P_\la)$ obtained is optimal because of the behavior of the solution near $\partial \Om$ since we showed that any weak solution of $(P_\la)$ lies in $C_{\phi_q}^+(\Om)$.
\end{Remark}
\begin{Remark}
 It follows from Theorem \ref{gm-reg2} that the extremal solution (when $\lambda=\Lambda$), in case of critical growth nonlinearities is a classical solution which extends the results in \cite{AJS} where in this regard only subcritical nonlinearities are considered.
\end{Remark}}
\section{Preliminaries}
We start with defining the function spaces. Given any $\phi \in C_0(\overline \Om)$ such that $\phi >0$ in $\Om$ we define
\[{C_{\phi}(\Om):= \{u \in C_0(\overline{\Om})| \;\exists \;c\geq0 \;\text{such that}\; |u(x)|\leq c\phi(x), \; \forall x \in \Om \}}\]
with the usual norm $\displaystyle \left\|\frac{u}{\phi}\right\|_{L^\infty(\Om)}$ and the associated positive cone. We define the following open convex subset of $C_{\phi}(\Om)$ as
\[C_{\phi}^+(\Om):= \left\{ u \in C_{\phi}(\Om)|\; \inf_{x\in \Om}\frac{u(x)}{\phi(x)}>0 \right\}.\]
In particular, $C_{\phi}^+$ contains all those functions $u \in C_0(\Om)$ with $k_1\phi \leq u\leq k_2 \phi$ in $\Om$ for some $k_1,k_2>0$. The space $X$ is defined as
\[X= \left\{u|\;u:\mb R^n \ra\mb R \;\text{is measurable},\;
u|_{\Om} \in L^2(\Om)\;
 \text{and}\;  \frac{(u(x)- u(y))}{ |x-y|^{\frac{n}{2}+s}}\in
L^2(Q)\right\},\]
\noi where $Q=\mb R^{2n}\setminus(\mc C\Om\times \mc C\Om)$ and
 $\mc C\Om := \mb R^n\setminus\Om$  endowed with the norm
 \[\|u\|_X = \|u\|_{L^2(\Om)} + \left[u\right]_X,\]
 where
 \[\left[u\right]_X= \left( \int_{Q}\frac{|u(x)-u(y)|^{2}}{|x-y|^{n+2s}}\,\mathrm{d}x\mathrm{d}y\right)^{\frac12}=\left(\frac{1}{C^n_s}\int_\Omega u(-\Delta)^su\,\mathrm{d}x\mathrm{d}y\right)^{\frac12}.\]
 Then we define $ X_0 = \{u\in X : u = 0 \;\text{a.e. in}\; \mb R^n\setminus \Om\}$ which forms a Hilbert space with the inner product defined as
 \[\langle u ,v\rangle := C^n_s\int_Q \frac{(u(x)-u(y))(v(x)- v(y))}{|x-y|^{n+2s}}~dxdy.\]
The energy functional corresponding to $(P_\la)$ is given by $I_\la :X_0 \to \mb R$ defined as
\[I_\la(u) = \frac{C^n_s\|u\|^2}{2}- \frac{1}{1-q}\int_\Om |u|^{1-q}~dx -\frac{1}{2^*_s}\int_\Om |u|^{2^*_s}~dx.\]
Let {$q>0$ satisfies $q(2s-1)< (2s+1)$. Then for any $\varphi \in X_0$ and $u \in C_{\phi_q}^+(\Om)$, by Hardy's inequality, we obtain}
\begin{equation}\label{gm-11}
\int_{\Om}u^{-q}\varphi \leq \left(\int_\Om \frac{dx}{(\delta(x))^{\frac{2s(q-1)}{(q+1)}}}\right)^{\frac12}\left(\frac{\varphi^2}{(\delta(x))^{2s}}\right)^{\frac12}< K \|\varphi\|<+\infty
\end{equation}
where $K>0$ is a constant. If we define $D(I)= \{u\in X_0:\; I_\la(u)< \infty\}$ then by virtue of \eqref{gm-11} we get that $D(I)\neq \emptyset$. This gives an importance of the inequality $q(2s-1)<(2s+1)$. From { the proof of Theorem 1.2 in \cite{AJS}}, we know that if $0<q<1$ and $u \in X_0$ satisfies $u\geq c\delta^s$ then $I_\la$ is G\^{a}t{e}aux differentiable at $u$. In the preceding lemma, we show the same property of $I_\la$ when $q\geq 1$ satisfies $q(2s-1)<(2s+1)$.
\begin{Proposition}\label{gm-prop1}
If $M=\{u \in X_0:\; u_1\leq u\leq u_2\}$ where $u_1\in C_{\phi_q}^+(\Om)$ and $u_2 \in X_0$ then $I_\la$ is G\^{a}teaux differentiable at $u$ in the direction $(v-u)$ where $v,u \in M$.
\end{Proposition}
\begin{proof}
We need to show that
\begin{align*}
\lim_{t \to 0} \frac{I_\la(u+t(v-u))-I_\la(u)}{t} &= C^n_s\int_Q\frac{(v(x)-v(y))((v-u)(x)-(v-u)(y))}{|x-y|^{n+2s}}~dxdy\\
& \quad - \int_\Om u^{-q}(v-u)~dx- \la\int_{\Om}u^{2^*s-1}(v-u)~dx.
\end{align*}
It is enough to show this for the singular term; for the rest two terms, the proof is standard. For any $t \in (0,1)$, $u+t(v-u)\in M$ since $M$ is convex. Consider $F(u)= \displaystyle\frac{1}{1-q}\int_{\Om}u^{1-q}~dx$
then using mean value Theorem we get
\begin{align*}
\frac{F(u+t(v-u))-F(u)}{t}  &= \frac{1}{t(1-q)} \int_\Om \left((u+t(v-u))^{1-q}-u^{1-q}\right)(x)~dx \\
&= \int_\Om (u+t\theta(v-u))^{-q}(x)(v-u)(x)~dx
\end{align*}
for some $\theta \in (0,1)$. Since $(u+t\theta(v-u)) \in M$ {and \eqref{gm-11}}, we have
\[\int_\Om (u+t\theta(v-u))^{-q}(v-u)~dx  \leq \int_\Om u_1^{-q}(v-u)~dx < +\infty.\]
So using Lebesgue Dominated convergence theorem we pass through the limit $t \to 0$ and get
\[\lim_{t\to 0}\frac{F(u+t(v-u))-F(u)}{t} = \int_\Om u^{-q}(v-u)~dx.\]
This completes the proof. {\hfill {$\square$}\goodbreak \medskip}
\end{proof}

\noi Let {$L(u):= (-\De)^su-u^{-q}$} then $L$ forms a monotone operator. So we have the following comparison principle {following} Lemma $3.1$ of \cite{TJS-para}.
\begin{Lemma}\label{comp-princ-gm}
Let {$u_1,u_2 \in X_0\cap C^+_{\phi_q}(\Omega)$} are weak solutions of
\[L(u_1)=g_1\; \text{in}\; \Om,\;\; L(u_2)=g_2\; \text{in}\; \Om \]
with $g_1,g_2 \in L^2(\Om)$ such that $g_1\leq g_2$. Then $u_1\leq u_2$ a.e. in $\Om$. Moreover if $g \in L^\infty(\Om)$ then the problem
\[L(u)=g\; \text{in}\; \Om,\;\; u= 0 \; \text{in}\; \mb R^n\setminus \Om\]
has a unique solution in $X_0$.
\end{Lemma}

\section{Existence result}
 Let us define
\[\La:= \sup\{\la>0:\; (P_\la) \; \text{has a weak solution} \}.\]
Also let $w\in C_0(\overline \Om)$ solves the purely singular problem
\[(-\De)^s w = w^{-q},\;w >0\; \text{in}\; \Om, \; w= 0 \;\text{in}\; \mb R^n \setminus \Om.\]
Then Theorem $1.2$ and Theorem $1.4$ of \cite{AJS} gives us that $w$ is unique, $w \in X_0\cap C_{\phi_q}^+(\Om)$ and $w \in C^{\frac{2s}{q+1}}(\mb R^n)$.
 So we basically focus on the case $q\geq 1$ satisfying $q(2s-1)<(2s+1)$ because when $q \in (0,1)$, the case follows easily along the same line. In this context, next is an important Lemma for $\La$.

\begin{Lemma}\label{gm-lem1}
It holds $0< \La<+\infty$.
\end{Lemma}
\begin{proof}
First we prove that $\La<+\infty$. Using $\phi_{1,s}$ as the test function in $(P_\la)$ we get
\begin{equation}\label{gm-1}
 \int_{\Om} (u^{-q}\phi_{1,s}+ \la u^{2^*_s-1}\phi_{1,s})~dx = \int_{\mb R^n}\phi_{1,s}(-\De)^su~dx=\int_{\mb R^n}u(-\De)^s\phi_{1,s}~dx= \la_{1,s}\int_{\Om}u\phi_{1,s}~dx .
 \end{equation}
If we choose a $\la>0$ which satisfies {$ t^{-q}+\lambda t^{2^*_s-1}> 2 \la_{1,s}t$} for all $t>0$ then we get a contradiction to \eqref{gm-1}. Therefore it must be $\La<+\infty$. Now to prove $\La>0$ we need sub and supersolution for $(P_\la)$. It is easy to see that $\underline{u_\la}= w$ forms a subsolution of $(P_\la)$ and $\overline{u_\la}= \underline{u_\la}+Mz$ {for $\lambda>0$ small enough and for a $M=M(\lambda)>0$} forms a supersolution of $(P_\la)$, where $0<z\in X_0$ solves $(-\De)^sz=1$ in $\Om$. Now we define the closed convex subset $M_\la$ of $X_0$ as
\[M_\la := \{u \in X_0:\; \uline{u_\la}\leq u\leq \oline{u_\la}\}.\]
Consider the iterative scheme $(k\geq 1)$:
{\begin{equation*}
(P_{\la,k})\left\{
\begin{split}
(-\De)^su_k  -u_k^{-q}&= \la u_{k-1}^{2^*_s-1},\; u_k>0\; \text{in}\l \Om\\
u_k &= 0 \; \text{in} \; \mb R^n \setminus \Om
\end{split}
\right.
\end{equation*}}
with $u_0= \underline{u_\la}$. The existence of $\{u_k\}$ in $X_0\cap M_\la \cap C_{\phi_q}^+(\Om)$ can be proved by considering the approximated problem corresponding to $(P_{\la,k})$, for instance {we} refer Theorem $2.4$ of \cite{TJS-para}. From Lemma \ref{comp-princ-gm}, it follows that $\{u_k\}$ is increasing and $u_k \in M_\la$ for all $k$. Let $\lim\limits_{k \uparrow\infty}u_k = u_\la$. Then testing $(P_{\la,k})$ by $u_k$ we get
\[\|u_k\|^2 \leq 2\int_\Om \oline{u_\la}^2dx + \la\int_\Om \oline{u_\la}^{2^*_s}dx + \int_{\Om}\oline{u_\la}\uline{u_\la}^{-q} \leq K_\la  \]
where $K_\la>0$ is a constant depending on $\la$. So, up to a subsequence, $u_k \rightharpoonup u_{\la}$ in $X_0$. Finally using Lebesgue dominated convergence Theorem we pass through the limit in $(P_{\la,k})$ to obtain $u_\la$ solves $(P_\la)$ weakly and obviously, $u_\la \in M_\la$. This proves that $\La>0$.{\hfill {$\square$}\goodbreak \medskip}
\end{proof}

\noi Now, we prove the existence of a weak solution for $(P_\la)$ whenever $\la \in (0,\La)$.
\begin{Proposition}\label{gm-prop2}
For each $\la \in (0,\La)$, $(P_\la)$ admits a weak solution $w \in C_{\phi_q}^+(\Om)$.
\end{Proposition}
\begin{proof}
The proof goes along the line of Perron's method adapted over a nonlocal framework (see Lemma $2.2$ of \cite{haitao}). Let $\la \in (0,\La)$ and $\la^\prime \in (\la,\La)$ then it is easy to see that $u_{\la^\prime}$, a weak solution of $(P_{\la^\prime})$, forms a supersolution for $(P_{\la})$. Such a $\la^\prime$ exists because of the definition of $\La$ and Lemma \ref{gm-lem1}. Let $\uline{u_\la}$ be the same function as defined in Lemma \ref{gm-lem1} and consider the closed convex subset $W_\la$ of $X_0$ as
\[W_{\la}= \{u \in X_0:\; \uline{u_\la} \leq u\leq u_{\la^\prime}. \}\]
Then for each $u \in W_\la$, because of fractional Sobolev embedding $I_\la$ satisfies
\[I_\la(u) \geq \frac{C^n_s\|u\|^2}{2}-\frac{C}{2^*_s}\|u\|^{2^*_s}  \]
which implies that $I_\la$ is bounded from below and coercive over $W_\la$. If $\{u_k\} \subset W_\la$ be such that $u_k \rightharpoonup u_0$ in $X_0$ as $k \to \infty$ then since for each $k$, $u_k \geq \uline{u_\la}$ for $q>1$ and $u_k \leq {u_{\la^\prime}}$ for $q \in (0,1]$, $\textstyle \int_\Om u_k^{1-q}~dx \leq \int_\Om \uline{u_\la}^{1-q}~dx$, we can use Lebesgue Dominated convergence theorem to get that
\begin{equation*}\label{gm-2}
 \int_\Om u_k^{1-q}~dx \to \int_\Om u_0^{1-q}~dx\; \text{as}\; k \to \infty.
\end{equation*}
 Hence from weak lower semicontinuity of norms, it follows that $I_\la$ is weakly lower semicontinuous over $W_\la$. Moreover, $W_\la$ is weakly sequentially closed subset of $X_0$. Therefore there exists a $w \in W_\la$ such that
 \begin{equation}\label{gm-3}
 \inf_{u \in W_\la} I_\la(u) = I_\la(w).
 \end{equation}
 \noi \textbf{Claim-} $w$ is a weak solution of $(P_\la)$.\\
 Let $\varphi \in X_0$ and $\e>0$ then we define
 \[v_\e = \min\{u_{\la^\prime}, \max\{\uline{u_\la}, w+\e \varphi\} \} = w+\e \varphi - \varphi^\e+ \varphi_\e\]
 where $\varphi^\e = \max\{0, w+\e\varphi- u_{\la^\prime}\}$ and $\varphi_\e=\max\{0, \uline{u_\la}-w-\e\varphi\}$. By construction $v_\e \in W_\la$ and $\varphi^\e, \varphi_\e \in X_0 \cap L^\infty(\Om)$. Since $w+t(v_\e-w) \in W_\la$ for each $t\in(0,1)$, using \eqref{gm-3} and Proposition \ref{gm-prop1} we get that
 \begin{align*}
 0 &\leq \lim_{t\to 0^+} \frac{I_\la(w+t(v_\e-w))- I_\la(w)}{t} \\
  & = \int_Q (v_\e-w)(-\De)^sw ~dx- \int_{\Om}w^{-q}(v_\e-w)dx- \int_{\Om}w^{2^*_s-1}(v_\e-w)dx.
 \end{align*}
 This on simplification gives
 \begin{equation}\label{gm-4}
 \int_{\mb R^n}\varphi(-\De)^sw~dx {-} \int_{\Om}(w^{-q}+\la w^{2^*s-1})\varphi~dx \geq \frac{1}{\e} (E^\e-E_\e)
 \end{equation}
 where
 \begin{align*}
 E^\e &= \int_{\mb R^n}\varphi^\e(-\De)^sw dx- \int_{\Om}(w^{-q}+\la w^{2^*_s-1})\varphi^\e~dx\\
 &= \int_{\mb R^n}\varphi^\e(-\De)^s(w-u_{\la^\prime})dx + \int_{\mb R^n}\varphi^\e(-\De)^su_{\la^\prime}dx - \int_{\Om}(w^{-q}+\la w^{2^*_s-1})\varphi^\e~dx \\
 E_\e &= \int_{\mb R^n}\varphi_\e(-\De)^sw dx- \int_{\Om}(w^{-q}+\la w^{2^*_s-1})\varphi_\e~dx\\
 & =\int_{\mb R^n}\varphi_\e(-\De)^s(w-\uline{u_{\la}})dx + \int_{\mb R^n}\varphi_\e(-\De)^s\uline{u_{\la}}dx - \int_{\Om}(w^{-q}+\la w^{2^*_s-1})\varphi_\e~dx.
 \end{align*}
 We define $\Om^\e = \{x \in \Om :\; (w+\e\varphi)(x) \geq u_{\la^\prime}> w(x)\}$ so that $\mc L(\Om^\e) \to 0$ as $\e \to 0^+$ and also $\mc C \Om^\e {:=\Om\setminus\Om_\e} \subset \{x \in \Om:\; (w+\e \varphi)(x)< u_{\la^\prime}(x)\}$ which implies that $\mc L(\Om^\e \times \mc C\Om^\e) \to 0$ as $\e \to 0^+$. Now we consider the term
 \begin{align*}
&\int_{\mb R^n}\varphi^\e(-\De)^s(w-u_{\la^\prime})dx\\ &= \int_Q \frac{((w-u_{\la^\prime})(x)-(w-u_{\la^\prime})(y))(\varphi_\e(x)-\varphi_\e(y))}{|x-y|^{n+2s}}dxdy \\
& = \int_{\Om^\e}\int_{\Om^\e}\frac{|(w-u_{\la^\prime})(x)-(w-u_{\la^\prime})(y)|^2}{|x-y|^{n+2s}}dxdy\\
 &\quad + \e \int_{\Om^\e}\int_{\Om^\e}\frac{((w-u_{\la^\prime})(x)-(w-u_{\la^\prime})(y))(\varphi(x)-\varphi(y))}{|x-y|^{n+2s}}dxdy\\
 & \quad+2 \int_{\Om^\e}\int_{\mc C\Om^\e} \frac{(w-u_{\la^\prime})^2(x)}{|x-y|^{n+2s}}dxdy + 2\e\int_{\Om^\e}\int_{\mc C\Om^\e}\frac{(w-u_{\la^\prime})(x)\varphi(x)}{|x-y|^{n+2s}}dxdy\\
 & \quad - 2\int_{\Om^\e}\int_{\mc C\Om^\e}\frac{(w-u_{\la^\prime})(x)(w-u_{\la^\prime})(y)}{|x-y|^{n+2s}}dxdy +2\e \int_{\Om^\e}\int_{\mc C\Om^\e}\frac{(w-u_{\la^\prime})(y)\varphi(x)}{|x-y|^{n+2s}}dxdy\\
 & \quad +2 \int_{\Om^\e}\int_{\mc C\Om} \frac{(w-u_{\la^\prime})^2(x)}{|x-y|^{n+2s}}dxdy + 2\e\int_{\Om^\e}\int_{\mc C\Om}\frac{(w-u_{\la^\prime})(x)\varphi(x)}{|x-y|^{n+2s}}dxdy\\
 &\geq  \e \int_{\Om^\e}\int_{\Om^\e} \frac{((w-u_{\la^\prime})(x)-(w-u_{\la^\prime})(y))(\varphi(x)-\varphi(y))}{|x-y|^{n+2s}}dxdy \\
 & \quad  + 2\e\int_{\Om^\e}\int_{\mc C\Om^\e}\frac{(w-u_{\la^\prime})(x)\varphi(x)}{|x-y|^{n+2s}}dxdy- 2\e^2\int_{\Om^\e}\int_{\mc C\Om^\e}\frac{\varphi(x)\varphi(y)}{|x-y|^{n+2s}}dxdy \\
 & \quad +2\e \int_{\Om^\e}\int_{\mc C\Om^\e}\frac{(w-u_{\la^\prime})(y)\varphi(x)}{|x-y|^{n+2s}}dxdy + 2\e\int_{\Om^\e}\int_{\mc C\Om}\frac{(w-u_{\la^\prime})(x)\varphi(x)}{|x-y|^{n+2s}}dxdy\\
 \end{align*}
where to obtain the last inequality, we use the fact that if $(x,y) \in \Om^\e\times\mc C\Om^\e$ then $(w-u_{\la^\prime})(x)(w-u_{\la^\prime})(y)\leq \e^2\varphi(x)\varphi(y)$.
Therefore we get
\[\frac{1}{\e}\int_{\mb R^n}\varphi^\e(-\De)^s(w-u_{\la^\prime})dx \geq o(1) \; \text{as}\; \e \to 0^+.\]
Moreover using the fact that ${u_{\la^\prime}}$ is a supersolution of $(P_\la)$, the other terms of $\frac{1}{\e}E^\e$ can be estimated as
\begin{align*}
&\frac{1}{\e} \int_{\mb R^n}\varphi^\e(-\De)^su_{\la^\prime}~dx - \frac{1}{\e} \int_{\Om}(w^{-q}+\la w^{2^*s-1})\varphi^\e~dx \\
&\geq \frac{1}{\e}\int_{\Om^\e} (u_{\la^\prime}^{-q}-w^{-q})\varphi^\e~dx +  \frac{1}{\e}\int_{\Om^\e} (u_{\la^\prime}^{2^*_s-1}-w^{2^*_s-1})\varphi^\e ~dx\\
& \geq -\int_{\Om^\e} |u_{\la^\prime}^{-q}-w^{-q}||\varphi|dx =o(1)\; \text{as}\; \e \to 0^+.
\end{align*}
Altogether we get
\[\frac{1}{\e}E^\e \geq o(1) \; \text{as}\; \e \to 0^+\]
and similarly we obtain
\[\frac{1}{\e}E_\e \leq o(1) \; \text{as}\; \e \to 0^+.\]
Hence \eqref{gm-4} gives that for all $\varphi \in {X_0}$
\[ \int_{\mb R^n}\varphi(-\De)^sw~dx- \int_{\Om}(w^{-q}+\la w^{2^*_s-1})\varphi~dx \geq o(1)\; \text{as}\; \e \to 0^+\]
but since $\varphi$ was arbitrary, this implies that $w$ is a weak solution of $(P_\la)$. This establishes the proof. {\hfill {$\square$}\goodbreak \medskip}
\end{proof}

\noi We now prove a special property of $w$, the weak solution of $(P_\la)$ obtained in Proposition \ref{gm-prop2} following the proof of Proposition $3.5$ of \cite{peral-hardy}.
\begin{Lemma}\label{loc-min-gm}
Let $\la \in (0,\La)$ and $w$ denotes the weak solution of $(P_\la)$ obtained in Proposition \ref{gm-prop2}. Then $w$ forms a local minimum of the functional $I_\la$.
\end{Lemma}
\begin{proof}
We argue by contradiction, so suppose $w$ is not a local minimum of $I_\la$. Then there exists a sequence $\{u_k\}\subset X_0$ satisfying
\begin{equation}\label{gm-5}
\|u_k-w\| \to 0\; \text{as}\; k \to \infty \; \text{and}\; I_\la({u_k})< I_\la(w).
\end{equation}
We define $\uline u= \uline{u_\la}$ and $\oline{u}= u_{\la^\prime}$ as sub and supersolution of $(P_\la)$ as defined in the proof of Proposition \ref{gm-prop2}. Also  we define
\begin{equation*}
v_k = \max\{\uline{u}, \min\{u_k,\uline{u}\}\}=\left\{
\begin{split}
\uline{u},\; &\text{if}\; u_k<\uline{u},\\
u_k,\; &\text{if}\; \uline{u}\leq u_k\leq\oline{u},\\
\oline{u},\; &\text{if}\; u_k>\uline{u},
\end{split}
\right.
\end{equation*}
and $\uline{w_k}= (u_k-\uline{u})^-$, $\oline{w_k} = (u_k-\oline{u})^+$. Correspondingly, we define the sets $\uline{S_k}= \text{Supp}(\uline{w_k})$ and $\oline{S_k}= \text{Supp}(\oline{w_k})$. Then $u_k = v_k - \uline{w_k}+\oline{w_k}$ and $v_k \in W_\la$ where $W_\la$ has been defined in Proposition \ref{gm-prop2}. It follows that
\begin{align*}
\int_\Om (u_k^+)^{1-q}dx &= \int_{\uline{S_k}} (u_k^+)^{1-q}dx + \int_{\oline{S_k}}(u_k^+)^{1-q}dx+ \int_{\uline{u}\leq v_k\leq \oline{u}} (v_k)^{1-q}dx\\
& = \int_{\uline{S_k}} ((u_k^+)^{1-q}- \uline{u}^{1-q})dx + \int_{\oline{S_k}}((u_k^+)^{1-q}-\oline{u}^{1-q})dx+ \int_{\Om} (v_k)^{1-q}dx\\
\text{and}\;\int_\Om (u_k^+)^{2^*_s}dx & = \int_{\uline{S_k}} (u_k^+)^{2^*_s}dx + \int_{\oline{S_k}}(u_k^+)^{2^*_s}dx+ \int_{\uline{u}\leq v_k\leq \oline{u}} (v_k)^{2^*_s}dx\\
& = \int_{\uline{S_k}} ((u_k^+)^{2^*_s}- \uline{u}^{2^*_s})dx + \int_{\oline{S_k}}((u_k^+)^{2^*_s}-\oline{u}^{2^*_s})dx+ \int_{\Om} (v_k)^{2^*_s}dx.
\end{align*}
Then we can express $I_\la(u_k)$ as
\begin{equation}\label{gm-6}
\begin{split}
I_\la(u_k)&= I_{\la}(v_k)+ \frac{J_0}{2} - \frac{1}{1-q}\left(\int_{\uline{S_k}} ((u_k^+)^{1-q}- \uline{u}^{1-q})dx+
\int_{\oline{S_k}}((u_k^+)^{1-q}-\oline{u}^{1-q})dx\right)\\
 &\quad  -\frac{\la}{2^*_s}\left( \int_{\uline{S_k}} ((u_k^+)^{2^*_s}- \uline{u}^{2^*_s})dx + \int_{\oline{S_k}}((u_k^+)^{2^*_s}-\oline{u}^{2^*_s})dx\right)
\end{split}
\end{equation}
where $J_0= C^n_s(\|u_k\|^2-\|v_k\|^2)$. While denoting $S_k = \{x \in \Om:\; \uline{u}\leq v_k\leq \oline{u}\}$ and $h_k(x,y)= (u_k(x)-u_k(y))^2 -(v_k(x)-v_k(y))^2$, we get
\begin{align*}
J_0 &= \int_{\uline{S_k}}\int_{\uline{S_k}}\frac{h_k(x,y)}{|x-y|^{n+2s}}dxdy + \int_{\oline{S_k}}\int_{\oline{S_k}}\frac{h_k(x,y)}{|x-y|^{n+2s}}dxdy+ 2\int_{\uline{S_k}}\int_{\oline{S_k}}\frac{h_k(x,y)}{|x-y|^{n+2s}}dxdy\\
&+\quad + 2\int_{\uline{S_k}}\int_{S_k}\frac{h_k(x,y)}{|x-y|^{n+2s}}dxdy+ 2\int_{\oline{S_k}}\int_{S_k}\frac{h_k(x,y)}{|x-y|^{n+2s}}dxdy.
\end{align*}
Since $u_k = \oline{w_k}+ \oline{u}$ and $v_k = \oline{u}$ in $\oline{S_k}$ and $u_k = \uline{u}-\uline{w_k}$ and $v_k = \uline{u}$ in $\uline{S_k}$  we get that
\begin{align*}
\int_{\uline{S_k}}\int_{\uline{S_k}}\frac{h_k(x,y)}{|x-y|^{n+2s}}dxdy &= \int_{\uline{S_k}}\int_{\uline{S_k}}\frac{(\uline{w_k}(x)-\uline{w_k}(y))^2}{|x-y|^{n+2s}}dxdy\\
 & \quad- 2\int_{\uline{S_k}}\int_{\uline{S_k}}\frac{(\uline{w_k}(x)-\uline{w_k}(y))(\uline{u}(x)-\uline{u}(y))}{|x-y|^{n+2s}}dxdy\\
 \int_{\oline{S_k}}\int_{\oline{S_k}}\frac{h_k(x,y)}{|x-y|^{n+2s}}dxdy &= \int_{\oline{S_k}}\int_{\oline{S_k}}\frac{(\oline{w_k}(x)-\oline{w_k}(y))^2}{|x-y|^{n+2s}}dxdy\\
 & \quad+ 2\int_{\oline{S_k}}\int_{\oline{S_k}}\frac{(\oline{w_k}(x)-\oline{w_k}(y))(\oline{u}(x)-\oline{u}(y))}{|x-y|^{n+2s}}dxdy.\\
\end{align*}
Also similarly we obtain
\begin{align*}
\int_{\uline{S_k}}\int_{\oline{S_k}}\frac{h_k(x,y)}{|x-y|^{n+2s}}dxdy &=  \int_{\uline{S_k}}\int_{\oline{S_k}}\frac{(\uline{w_k}(x)+\oline{w_k}(y))^2}{|x-y|^{n+2s}}dxdy \\
&\quad -2 \int_{\uline{S_k}}\int_{\oline{S_k}}\frac{(\uline{w_k}(x)+\oline{w_k}(y))(\uline{u}(x)-\oline{u}(y))}{|x-y|^{n+2s}}dxdy,\\
\int_{\uline{S_k}}\int_{S_k}\frac{h_k(x,y)}{|x-y|^{n+2s}}dxdy &= \int_{\uline{S_k}}\int_{S_k}\frac{\uline{w_k}^2(x)}{|x-y|^{n+2s}}dxdy- 2 \int_{\uline{S_k}}\int_{S_k}\frac{\uline{w_k}(x)(\uline{u}(x)-u_k(y))}{|x-y|^{n+2s}}dxdy,\\
\int_{\oline{S_k}}\int_{S_k}\frac{h_k(x,y)}{|x-y|^{n+2s}}dxdy &= \int_{\oline{S_k}}\int_{S_k}\frac{\oline{w_k}^2(x)}{|x-y|^{n+2s}}dxdy+ 2 \int_{\oline{S_k}}\int_{S_k}\frac{\oline{w_k}(x)(\oline{u}(x)-u_k(y))}{|x-y|^{n+2s}}dxdy.
\end{align*}
Since $\mc C \uline{S_k}= \oline{S_k}\cup S_k$, $\mc C \oline{S_k}= \uline{S_k}\cup S_k$ and
\begin{align*}
\|\uline{w_k}\|^2 &= \int_{\uline{S_k}}\int_{\uline{S_k}}\frac{(\uline{w_k}(x)-\uline{w_k}(y))^2}{|x-y|^{n+2s}}dxdy + 2 \int_{\uline{S_k}}\int_{\mc C{S_k}}\frac{\uline{w_k}^2(x)}{|x-y|^{n+2s}}dxdy\\
\|\oline{w_k}\|^2 &= \int_{\oline{S_k}}\int_{\oline{S_k}}\frac{(\oline{w_k}(x)-\oline{w_k}(y))^2}{|x-y|^{n+2s}}dxdy + 2 \int_{\oline{S_k}}\int_{\mc C{S_k}}\frac{\oline{w_k}^2(x)}{|x-y|^{n+2s}}dxdy,
\end{align*}
using all above estimates, we can express $J_0$  as
\begin{align*}
J_0 &= C^n_s(\|\uline{w_k}\|^2+ \|\oline{w_k}\|^2) + 2 \left(\int_{\uline{S_k}}\int_{\oline{S_k}}\frac{(\uline{w_k}(x)+\oline{w_k}(y))^2}{|x-y|^{n+2s}}dxdy - \int_{\uline{S_k}}\int_{\oline{S_k}}\frac{\uline{w_k}^2(x)}{|x-y|^{n+2s}}\right.\\
 &\quad \left.- \int_{\oline{S_k}}\int_{\uline{S_k}}\frac{\oline{w_k}^2(x)}{|x-y|^{n+2s}}dxdy \right) - 2\int_{\uline{S_k}}\int_{\uline{S_k}}\frac{(\uline{w_k}(x)-\uline{w_k}(y))(\uline{u}(x)-\uline{u}(y))}{|x-y|^{n+2s}}dxdy\\
 & \quad + 2\int_{\oline{S_k}}\int_{\oline{S_k}}\frac{(\oline{w_k}(x)-\oline{w_k}(y))(\oline{u}(x)-\oline{u}(y))}{|x-y|^{n+2s}}dxdy -4 \int_{\uline{S_k}}\int_{\oline{S_k}}\frac{(\uline{w_k}(x)+\oline{w_k}(y))(\uline{u}(x)-\oline{u}(y))}{|x-y|^{n+2s}}dxdy\\
 &\quad  -4 \int_{\uline{S_k}}\int_{S_k}\frac{\uline{w_k}(x)(\uline{u}(x)-u_k(y))}{|x-y|^{n+2s}}dxdy +4 \int_{\oline{S_k}}\int_{S_k}\frac{\oline{w_k}(x)(\oline{u}(x)-u_k(y))}{|x-y|^{n+2s}}dxdy.
\end{align*}
Now we notice that if $(x,y)\in \uline{S_k}\times S_k$ then $(\uline u(x)-u_k(y))\leq (\uline u(x)-\uline u(y))$, if $(x,y)\in \oline{S_k}\times S_k$ then $(\oline u(x)-u_k(y))\geq (\oline u(x)-\oline u(y))$ and
\begin{align*}&\int_{\uline{S_k}}\int_{\oline{S_k}}\frac{(\uline{w_k}(x)+\oline{w_k}(y))^2}{|x-y|^{n+2s}}dxdy - \int_{\uline{S_k}}\int_{\oline{S_k}}\frac{\uline{w_k}^2(x)}{|x-y|^{n+2s}}dxdy - \int_{\oline{S_k}}\int_{\uline{S_k}}\frac{\oline{w_k}^2(x)}{|x-y|^{n+2s}}dxdy \\
&= 2 \int_{\oline{S_k}}\int_{\uline{S_k}}\frac{\oline{w_k}(x)\uline{w_k}(y)}{|x-y|^{n+2s}}dxdy.
\end{align*}
Also using change of variables, we have
\begin{align*}
&\int_{\uline{S_k}}\int_{\oline{S_k}}\frac{(\uline{w_k}(x)+\oline{w_k}(y))(\uline{u}(x)-\oline{u}(y))}{|x-y|^{n+2s}}dxdy\\
& =  \int_{\uline{S_k}}\int_{\oline{S_k}}\frac{\uline{w_k}(x)(\uline{u}(x)-\oline{u}(y))}{|x-y|^{n+2s}}dxdy - \int_{\oline{S_k}}\int_{\uline{S_k}}\frac{\oline{w_k}(x)(\oline{u}(x)-\uline{u}(y))}{|x-y|^{n+2s}}dxdy.
\end{align*}
Therefore altogether we obtain
\begin{align*}
J_0 &\geq C^n_s(\|\uline{w_k}\|^2 + \|\oline{w_k}\|^2) + 4 \int_{\oline{S_k}}\int_{\uline{S_k}}\frac{\oline{w_k}(x)\uline{w_k}(y)}{|x-y|^{n+2s}}dxdy + 2\int_{\mb R^n}\oline{w_k}(-\De)^s \oline{u}~dx -2 \int_{\mb R^n}\uline{w_k}(-\De)^s \uline{u}~dx\\
&  \quad -4 \int_{\oline{S_k}}\int_{\uline{S_k}}\frac{\oline{w_k}(x)(\oline{u}(x)-\uline{u}(y))}{|x-y|^{n+2s}}dxdy + 4 \int_{\uline{S_k}}\int_{\oline{S_k}}\frac{\uline{w_k}(x)(\uline{u}(x)-\oline{u}(y))}{|x-y|^{n+2s}}dxdy\\
& \quad -4 \int_{\uline{S_k}}\int_{\oline{S_k}}\frac{(\uline{w_k}(x)+\oline{w_k}(y))(\uline{u}(x)-\oline{u}(y))}{|x-y|^{n+2s}}dxdy\\
& \geq C^n_s( \|\uline{w_k}\|^2 + \|\oline{w_k}\|^2) +  2\int_{\mb R^n}\oline{w_k}(-\De)^s \oline{u}~dx -2 \int_{\mb R^n}\uline{w_k}(-\De)^s \uline{u}~dx
\end{align*}
where we used the fact that if $(x,y)\in \oline{S_k}\times \uline{S_k}$ then $\oline{w_k}(x)\uline{w_k}(y)\geq 0$.
Now recalling that $\uline{u}$ and $\oline{u}$ forms sub and supersolution of $(P_\la)$ respectively, inserting the above inequality in \eqref{gm-6} we obtain
\begin{equation}
\begin{split}
I_\la(u_k) &\geq I_\la(v_k) + \frac{C^n_s\|\uline{w_k}\|^2}{2}+\frac{C^n_s\|\oline{w_k}\|^2}{2} + \int_{\oline{S_k}}\left( \frac{\oline{u}^{1-q}-(\oline{u}+\oline{w_k})^{1-q}}{1-q} + \oline{u}^{-q}\oline{w_k}\right)dx\\
& \quad + \int_{\uline{S_k}}\left( \frac{\uline{u}^{1-q}-(\uline{u}-\uline{w_k})^{1-q}}{1-q} - \uline{u}^{-q}\uline{w_k}\right)dx +\la \int_{\oline{S_k}}\left( \frac{\oline{u}^{2^*_s}-(\oline{u}+\oline{w_k})^{2^*_s}}{2^*_s} + \oline{u}^{2^*_s-1}\oline{w_k}\right)dx\\
& \quad + \la\int_{\uline{S_k}}\left( \frac{\uline{u}^{2^*_s}-(\uline{u}-\uline{w_k})^{2^*_s}}{2^*_s} - \uline{u}^{2^*_s-1}\uline{w_k}\right)dx.
\end{split}
\end{equation}
Now from mean value Theorem it follows that there exists $\theta\in (0,1)$ (where $\theta$ may change its value for different function below) such that
\begin{equation}\label{gm-7}
\begin{split}
I_{\la}(u_k) & \geq I_\la(v_k) + \frac{C^n_s\|\uline{w_k}\|^2}{2}+\frac{C^n_s\|\oline{w_k}\|^2}{2}- \int_{\oline{S_k}}( (\oline{u}+\theta\oline{w_k})^{-q}- \oline{u}^{-q})\oline{w_k}dx\\
& \quad  - \int_{\uline{S_k}}(  \uline{u}^{-q} - (\uline{u}+\theta\uline{w_k})^{-q})\uline{w_k}dx - \la \int_{\oline{S_k}}( (\oline{u}+\theta\oline{w_k})^{2^*_s-1}- \oline{u}^{2^*_s-1})\oline{w_k}dx\\
& \quad -\la \int_{\uline{S_k}}(  \uline{u}^{2^*_s-1} - (\uline{u}+\theta\uline{w_k})^{2^*_s-1})\uline{w_k}dx\\
& \geq  I_\la(v_k) + \frac{C^n_s\|\uline{w_k}\|^2}{2}+\la \int_{\oline{S_k}}( (\oline{u}+\theta\oline{w_k})^{2^*_s-1}- \oline{u}^{2^*_s-1})\oline{w_k}dx \\
& \quad-\la \int_{\uline{S_k}}(  \uline{u}^{2^*_s-1} - (\uline{u}+\theta\uline{w_k})^{2^*_s-1})\uline{w_k}dx.
\end{split}
\end{equation}
Now since $2^*_s>2$, there exists constant $C>0$ such that \eqref{gm-7} reduces to
\begin{equation}\label{gm-8}
\begin{split}
I_\la(u_k) &\geq I_\la(v_k) + \frac{C^n_s\|\uline{w_k}\|^2}{2}+\frac{C^n_s\|\oline{w_k}\|^2}{2} - C \int_{\uline{S_k}}(\uline{u}^{2^*_s-2}- \uline{w_k}^{2^*_s-2} )\uline{w_k}^2dx \\
&\quad - C \int_{\oline{S_k}}(\oline{u}^{2^*_s-2}- \oline{w_k}^{2^*_s-2} )\oline{w_k}^2dx\\
&\geq  I_\la(v_k) + \frac{C^n_s\|\uline{w_k}\|^2}{2}+\frac{C^n_s\|\oline{w_k}\|^2}{2}  - C \left(\int_{\uline{S_k}}|\uline{u}|^{2^*_s}\right)^{\frac{2^*_s-2}{2^*_s}}\|\uline{w_k}\|^2\\
&\quad - C \left(\int_{\oline{S_k}}|\oline{u}|^{2^*_s}\right)^{\frac{2^*_s-2}{2^*_s}}\|\oline{w_k}\|^2 - C\|\oline{w_k}\|^{2^*_s}.
\end{split}
\end{equation}
\textbf{Claim-} $\lim\limits_{k \to \infty}|\oline{S_k}|=0$ and $\lim\limits_{k \to \infty}|\uline{S_k}|=0$.\\
Let $\alpha>0$ and define
\begin{align*}
A_k = \{x \in \Om:\; u_k \geq \oline{u} \; \text{and}\; \oline{u}>w+\alpha\}, \; &\hat{A}_k = \{x \in \Om:\; u_k \leq \uline{u} \; \text{and}\; \uline{u}<w-\alpha\}\\
B_k = \{x \in \Om:\; u_k \geq \oline{u} \; \text{and}\; \oline{u}\leq w+\alpha\}, \; &\hat{B}_k = \{x \in \Om:\; u_k \leq \uline{u} \; \text{and}\; \uline{u}\geq w-\alpha\}.
\end{align*}
Since
\[0=\mc L (\{x\in \Om:\; \oline{u}< {w}\}) = \mc L(\cap_{j=1}^{\infty}\{x\in \Om:\; \oline{u}< {w}+ \frac{1}{j}\})\]
so there exists $j_0\geq 1$ large enough and $\alpha< 1/j_0$ such that $\mc L(\{x\in \Om:\; \oline{u}<w+\alpha\})\leq \e/2$. This implies that $\mc L(B_k)\leq \e/2$ and similarly, we obtain $\mc L(\hat{B}_k) \leq \e/2$. From \eqref{gm-5} we already have $|u_k -w|_2 \to 0$ as $k \to \infty$. So for $k \geq k_0$ large enough we get that
\[\frac{\alpha^2\e}{2} \geq \int_\Om|u_k-w|^2~dx \geq \int_{A_k}|u_k-w|^2~dx\geq \alpha^2\mc L(A_k)\]
which implies that $\mc L(A_k)\leq \frac{\e}{2}$ for $k \geq k_0$. Similarly we obtain $\mc L(\hat{A}_k)\leq \frac{\e}{2}$ for $k \geq k_0$. Now since $\oline{S_k} \subset A_k \cap B_k$ and $\uline{S_k} \subset \hat{A}_k \cap \hat{B}_k$ we get that $\mc L(\oline{S_k}) \leq \e$ and $\mc L(\uline{S_k})\leq \e$ for $k \geq k_0$. This proves the claim. Thus
\[\left(\int_{\oline{S_k}}|\oline{u}|^{2^*_s}\right)^{\frac{2^*_s-2}{2^*_s}} \leq o(1)\; \text{and}\; \left(\int_{\uline{S_k}}|\uline{u}|^{2^*_s}\right)^{\frac{2^*_s-2}{2^*_s}} \leq o(1)\]
which imposing in \eqref{gm-8} gives that for large enough $k$
\[I_\la(u_k) {\geq} I_\la(v_k) \geq I_{\la}(w)\]
which is a contradiction to \eqref{gm-5}. Therefore $w$ must be a local minimum of $I_\la$ over $X_0$. {\hfill {$\square$}\goodbreak \medskip}
\end{proof}

\begin{Theorem}\label{firstsol-gm}
There exists a positive weak solution of $(P_\La)$.
\end{Theorem}
\begin{proof}
 Let $\la_m\uparrow\La$  as $m \to \infty$ and $\{u_{\la_m}\}$ be a sequence of positive weak solutions to $(P_{\la_m})$, such that
$u_{\la_m}$ forms the local minimum of $I_{\la_m}$ as seen in Lemma \ref{loc-min-gm}. Since we consider the minimal solutions, we get $u_m\leq u_{m+1}$ for each $m$. Then, it is easy to see that $I_{\la_m}<0$ in the case $0<q<1$ whereas there exists a constant $K$ independent of $m$ such that $I_{\la_m}\leq K$ for all $m$ when $q>1$ but $q(2s-1)<(2s+1)$. This implies that $\{u_{\la_m}\}$ is uniformly bounded in $X_0$. {Therefore, up to a subsequence there exists $u_\La\in X_0$ such that} $u_{\la_m}\rightharpoonup u_\La$ weakly and pointwise a.e. in $X_0$ as $m \rightarrow \infty$. Also by construction $u_{\la_m} \geq \underline{u_{\la_1}}$ as defined in Lemma \ref{gm-lem1}. Therefore, $u_\La$ is a positive weak solution of $( P_\La)$.  {\hfill {$\square$}\goodbreak \medskip}
\end{proof}

\section{Multiplicity result}
We have already obtained the first solution for $(P_\la)$ in the previous section when $\la \in (0,\La)$ in $X_0$-topology. We fix $\la \in (0,\La)$ and let $w$ denotes the first weak solution of $(P_\la)$ obtained in Proposition \ref{gm-prop2}. 
In this section, we prove the existence of second solution of $( P_\la)$ using the machinery of mountain pass Lemma and {with the help of} Ekeland variational principle. Let us define the set
\[T = \{x \in X_0:\; u\geq w\; \text{a.e. in}\; \Om \}\]
and since $w$ forms a local minimizer of $I_\la$ we get that $I_\la(u) \geq  I_\la(w)$ whenever $\|u-w\|\leq \sigma_0$, for some constant $\sigma_0>0$. Then one of the following cases holds
\begin{enumerate}
\item[(ZA)](Zero Altitude) $\inf \{ I_\la(u)|\; u \in T, \;\|u-w\|=\sigma\}=  I_\la(w)$ for all $\sigma\in (0,\sigma_0)$.
\item[(MA)](Mountain Pass) There exists a $\sigma_1 \in (0,\sigma_0)$ such that $\inf\{ I_\la(u)|\; u \in T,\; \|u-w\|=\sigma_1\}> I_\la(w)$.
\end{enumerate}

\begin{Lemma}\label{sec-sol-ZA}
Let $(ZA)$ holds then there exists a $v \in T$ which solves $(P_\la)$ weakly and $\|v-w\| = \sigma$ for all $\sigma \in (0,\sigma_0)$.
\end{Lemma}
\begin{proof}
We follow the proof of Lemma $2.6$ of \cite{haitao} in a nonlocal framework. We fix $\sigma \in (0, \sigma_0)$ and $r>0$ such that $\sigma -r>0$ and $\sigma +r< \sigma_0$. Let us define the set
\[W= \{u \in T|\; 0<\sigma-r \leq \|u-w\|\leq \sigma+r\}\]
which is closed in $X_0$ and by $(ZA)$, $\inf\limits_{u \in W} I_\la(u)=  I_\la(w)$. So using Ekeland variational principle, for any minimizing sequence $\{u_k\}\subset X_0$ satisfying $\|u_k\|= \sigma$ and $ I_\la(u_k)\leq I_\la(w)+ \frac{1}{k}$, we get another sequence $\{v_k\}\subset W$ such that
\begin{equation}\label{sec-sol-gm1}
\left\{\begin{split}
& I_\la(v_k) \leq  I_\la(u_k)\leq I_\la(w)+ \frac{1}{k}\\
&\|u_k-v_k\| \leq \frac{1}{k}\\
& I_\la(v_k) \leq I_\la(z)+\frac{1}{k}\|z-v_k\|,\; \text{for all}\; z \in W.
\end{split}\right.
\end{equation}
We can choose $\e >0$ small enough so that $v_k + \e(z-v_k) \in W$ for $z \in T$. So from \eqref{sec-sol-gm1} we obtain
\[\frac{ I_\la(v_k + \e(z-v_k)) - I_\la(v_k)}{\e} \geq -\frac{1}{k}\|z-v_k\|.\]
Letting $\e \to 0^+$ and using the fact that $v_k\geq w$ for each $k$, for $z \in T$ we get
\begin{equation}\label{sec-sol-gm2}
\int_{\mb R^n}(-\De)^sv_k(z-v_k)-\int_\Om v_k^{-q}(z-v_k)~dx -\la\int_\Om v_k^{2^*_s-1}(z-v_k)~dx\geq -\frac{1}{k}\|z-v_k\|.
\end{equation}
Now since $\{v_k\}$ forms a bounded sequence in $X_0$, we get that there exists a $v \in X_0$ such that, up to a subsequence, $v_k \rightharpoonup v$ weakly in $X_0$ and pointwise a.e. in $\Om $ as $k \to \infty$. Since $v_k\geq w$ for each $k$, we get $v\geq w$ a.e. in $\Om$. In what follows, we will prove that $v$ is a weak solution of $(P_\la)$. For $\phi \in X_0$ and $\e>0$, we set $\phi_{k,\e} = (v_k+\e\phi-w)^- \in X_0$ which implies that $(v_k +\e\phi+\phi_{k,\e}) \in T$. Putting $z = v_k +\e\phi+\phi_{k,\e}$ in \eqref{sec-sol-gm2} we get
\begin{equation}\label{sec-sol-gm3}
\begin{split}
&C^n_s\int_Q \frac{(v_k(x)-v_k(y))((\e\phi+\phi_{k,\e})(x)-(\e\phi+\phi_{k,\e})(y))}{|x-y|^{n+2s}}~dxdy -  \int_{\Om}v_k^{-q}(\e\phi+\phi_{k,\e})~dx\\
&\quad \quad-{\la} \int_\Om v_k^{2^*_s-1}(\e\phi+\phi_{k,\e})~dx \geq \frac{-1}{k}\|(\e\phi+\phi_{k,\e})\|.
\end{split}
\end{equation}
We define the sets $\Om_{k,\e}= \text{Supp}\;\phi_{k,\e}$, $\Om_\e= \text{Supp}\;\phi_\e$ and $\Om_0=\{x \in \Om:\; v(x)=w(x)\}$. Then we get that $\mc L(\Om_\e\setminus \Om_0) \to 0$ as $\e \to 0$ and $\mc L(\Om_{k,\e}\setminus \Om_\e)+ \mc L(\Om_\e\setminus\Om_{k,\e}) \to 0$ as $k \to \infty$. Also since $|\phi_{k,\e}| \leq w+\e|\phi|$, using Lebesgue Dominated convergence theorem we get $\phi_{k,\e} \to \phi_\e= (v+\e\phi -w)^-$ in $L^m(\Om)$ for all $m \in [1,2^*_s]$. Moreover $\phi_{k,\e} \rightharpoonup \phi_\e$ weakly in $X_0$ and pointwise a.e. in $\Om$ as $k \to \infty$. {Now we estimate the following integral
\begin{equation}\label{sec-sol-gm4}
\begin{split}
&\int_Q \frac{(v_k(x)-v_k(y))(\phi_{k,\e}(x)-\phi_{k,\e}(y))}{|x-y|^{n+2s}}~dxdy \\
&= \int_Q \frac{(v_k(x)-v_k(y))(\phi_{\e}(x)-\phi_{\e}(y))}{|x-y|^{n+2s}}~dxdy\\
 &\quad+ \int_Q \frac{(v_k(x)-v_k(y))((\phi_{k,\e}-\phi_\e)(x)-(\phi_{k,\e}-\phi_\e)(y))}{|x-y|^{n+2s}}~dxdy := I_1 + I_2.
\end{split}
\end{equation}
We show that $I_2 \leq o_k(1)$ for which we split the integrals and estimate them separately. Let $H_k = \Om_{k,\e}\cap \Om_\e$ and $G_k= \Om_{k,\e}\setminus \Om_\e \cup \Om_\e\setminus \Om_{k,\e}$. Then
\begin{equation}\label{new1}
\begin{split}
&\int_\Om \int_{\mc C \Om} \frac{(v_k(x)-v_k(y))((\phi_{k,\e}-\phi_\e)(x)-(\phi_{k,\e}-\phi_\e)(y))}{|x-y|^{n+2s}}\\
&\leq \int_{H_k}\int_{\mc C \Om} \frac{v(x)(v-v_k)(x)}{|x-y|^{n+2s}} + \int_{G_k}\int_{\mc C \Om} \frac{v_k(x)(\phi_{k,\e}-\phi_\e)(x)}{|x-y|^{n+2s}}\\
& \leq \int_{H_k}\int_{\mc C \Om} \frac{v(x)(v-v_k)(x)}{|x-y|^{n+2s}} +\int_{G_k}\int_{\mc C \Om} \frac{v_k(x)\phi_{k,\e}(x)}{|x-y|^{n+2s}}\\
& =\int_{H_k}\int_{\mc C \Om} \frac{v(x)(v-v_k)(x)}{|x-y|^{n+2s}} + o_k(1)
\end{split}
\end{equation}
using the fact that $\mc L(\Om_{k,\e}\setminus \Om_\e)+ \mc L(\Om_\e\setminus\Om_{k,\e}) \to 0$ as $k \to \infty$ and Lebesgue Dominated convergence theorem. Similarly
\begin{equation}\label{new2}
\begin{split}
&\int_\Om \int_{ \Om} \frac{(v_k(x)-v_k(y))((\phi_{k,\e}-\phi_\e)(x)-(\phi_{k,\e}-\phi_\e)(y))}{|x-y|^{n+2s}}\\
&\leq \int_{H_k}\int_{H_k}\frac{(v(x)-v(y))((v-v_k)(x)-(v-v_k)(y))}{|x-y|^{n+2s}}\\
 &\quad+ 2\int_{H_k}\int_{G_k} \frac{(v_k(x)-v_k(y))((\phi_{k,\e}-\phi_\e)(x)-(\phi_{k,\e}-\phi_\e)(y))}{|x-y|^{n+2s}}\\
&\quad \quad+ \int_{G_k}\int_{G_k} \frac{(v_k(x)-v_k(y))((\phi_{k,\e}-\phi_\e)(x)-(\phi_{k,\e}-\phi_\e)(y))}{|x-y|^{n+2s}}\\
&\leq \int_{H_k}\int_{H_k}\frac{(v(x)-v(y))((v-v_k)(x)-(v-v_k)(y))}{|x-y|^{n+2s}} + o_k(1)
\end{split}
\end{equation}
using again the Lebesgue Dominated convergence theorem with the fact that $v_k-v \to 0$ and $\phi_{k,\e} -\phi_\e\to 0$ pointwise as $k \to \infty$. Combining \eqref{new1} and \eqref{new2} we obtain that
\begin{align*}
I_2 \leq \int_{H_k}\int_{H_k\cup \mc C\Om}\frac{(v(x)-v(y))((v-v_k)(x)-(v-v_k)(y))}{|x-y|^{n+2s}} +o_k(1) =o_k(1).
\end{align*}}
Therefore using this in \eqref{sec-sol-gm4}, we obtain
\[\int_Q \frac{(v_k(x)-v_k(y))(\phi_{k,\e}(x)-\phi_{k,\e}(y))}{|x-y|^{n+2s}}~dxdy \leq \int_Q \frac{(v_k(x)-v_k(y))(\phi_{\e}(x)-\phi_{\e}(y))}{|x-y|^{n+2s}}~dxdy +o_k(1).\]
Moreover, we have that $|v_k^{-q}(\e\phi+\phi_{k,\e})|\leq w^{-q}(w+2\e\phi) \in L^1(\Om)$ using the Hardy's inequality. Thus using Lebesgue Dominated convergence theorem and
passing on the limits $k \to \infty$ in \eqref{sec-sol-gm3} we get
\[0 \leq C^n_s\int_Q \frac{(v_k(x)-v_k(y))((\e\phi+\phi_{\e})(x)-(\e\phi+\phi_{\e})(y))}{|x-y|^{n+2s}}~dxdy- \int_\Om (v^{-q}+\la v^{2^*_s-1})(\e\phi+\phi_\e)~dx.  \]
Using the fact that $w$ is a weak solution of $(P_\la)$ and $v\geq w$, the above inequality implies that
\begin{align*}
&C^n_s\int_Q \frac{(v(x)-v(y))(\phi(x)-\phi(y))}{|x-y|^{n+2s}}~dxdy- \int_\Om v^{-q}\phi~dx - \la\int_\Om v^{2^*_s-1}\phi~dx\\
&\geq -\frac{1}{\e} \left( C^n_s\int_Q \frac{(v(x)-v(y))(\phi_{\e}(x)-\phi_{\e}(y))}{|x-y|^{n+2s}}~dxdy- \int_\Om v^{-q}\phi_\e~dx - \la\int_\Om v^{2^*_s-1}\phi_\e~dx \right)\\
& \geq \frac{1}{\e} \left(C^n_s \int_Q \frac{((w-v)(x)-(w-v)(y))(\phi_{\e}(x)-\phi_{\e}(y))}{|x-y|^{n+2s}}~dxdy + \int_{\Om}(v^{-q}-w^{-q})\phi_\e~dx\right)\\
& \geq C^n_s\int_{\Om_\e}\int_{\Om_\e} \frac{((v-w)(x)-(v-w)(y))(\phi(x)-\phi(y))}{|x-y|^{n+2s}}~dxdy\\
 &\quad +2C^n_s \int_{\Om_\e}\int_{\{w\leq v+\e\phi\}} \frac{((v-w)(x)-(v-w)(y))\phi(x)}{|x-y|^{n+2s}}~dxdy \\
&\quad +2C^n_s\int_{\Om_\e}\int_{\mc C\Om} \frac{(v-w)(x)\phi(x)}{|x-y|^{n+2s}}~dxdy + \int_{\Om_\e}(v^{-q}-w^{-q})\phi~dx\\
&=o(1)\;\text{as}\; \e \to 0^+
\end{align*}
using the fact that $|\Om_\e\setminus\Om_0|\to 0$ as $\e\to 0^+$. From this, we get that
\[C^n_s\int_Q \frac{(v(x)-v(y))(\phi(x)-\phi(y))}{|x-y|^{n+2s}}~dxdy- \int_\Om v^{-q}\phi~dx - \la\int_\Om v^{2^*_s-1}\phi~dx=0\; \text{for all}\; \phi \in X_0.\]
\textbf{Claim-} The sequence $v_k\to v$ strongly in $X_0$ as $k \to \infty$.\\
From Brezis Leib lemma we have
\begin{align*}
\|v_k\|^2-\|v_k-v\|^2 &=\|v\|^2+o(1)\\
\int_\Om |v_k|^{2^*_s}~dx -\int_\Om |v_k-v|^{2^*_s}~dx  &= \int_\Om |v|^{2^*_s}~dx +o(1).
\end{align*}
Since $v_k,v\geq w$ a.e. in $\Om$, we get
\[\int_\Om |v_k|^{1-q}~dx -\int_\Om |v|^{1-q}~dx = \int_\Om (v_k+\theta v)^{-q}(v_k-v)~dx, \; \text{for} \; \theta\in [0,1].\]
We know that $(v_k+\theta v)^{-q}(v_k-v) \to 0$ pointwise a.e. in $\Om$ and for any $E\subset \Om$, we have
\begin{equation}\label{sec-sol-gm5}
\int_\Om (v_k+\theta v)^{-q}(v_k-v)~dx \leq C \|\delta^{\frac{(1-q)s}{1+q}}(x)\|_{L^2(E)} \|v_k-v\|, \; \text{using Hardy's inequality}.
\end{equation}
Since $q(2s-1)<(2s+1)$, for any $\e>0$, there exists a $\rho>0$ such that $\|\delta^{\frac{(1-q)s}{1+q}}(x)\|_{L^2(E)}<\e$ whenever $\mc L(E)<\rho$. Hence from \eqref{sec-sol-gm5} and Vitali's convergence theorem we obtain
\[ \int_\Om (v_k+\theta v)^{-q}(v_k-v)~dx \to 0 \;\text{as}\; k \to \infty\]
that is
\[\int_\Om |v_k|^{1-q}~dx \to \int_\Om |v|^{1-q}~dx \;\text{as}\; k \to \infty.\]
Now the rest of the proof follows exactly as the proof of Lemma $2.6$ of \cite{haitao}. {\hfill {$\square$}\goodbreak \medskip}
 \end{proof}

\noi We define
 \begin{equation*}
 {S_s} = \inf_{u \in X_0\setminus \{0\}} \displaystyle \frac{\int_Q \displaystyle\frac{|u(x)-u(y)|^2}{|x-y|^{n+2s}}{\,\mathrm{d}x\mathrm{d}y}}{\left(\int_{\Om}|u|^{2^*_s}\,\mathrm{d}x\right)^{2/2^*_s}}
 \end{equation*}
as the best constant for the embedding $X_0 \hookrightarrow L^{2^*_s}(\Om)$. Consider the family of minimizers $\{U_{\epsilon}\}$ of $S_s$ (refer \cite{sv}) defined as
\[ U_{\epsilon}(x) = \epsilon^{-\frac{(n-2s)}{2}}\; u^*\left(\frac{x}{\epsilon}\right),\; x \in \mb R^n \]
where $u^*(x) = \bar{u}\left(\frac{x}{S_s^{\frac{1}{2s}}}\right),\; \bar{u}(x) = \frac{\tilde{u}(x)}{\vert u \vert_{2^*_s}}$ and $\tilde{u}(x)= \alpha(\beta^2 + |x|^2)^{-\frac{n-2s}{2}}$ with $\alpha \in \mb R \setminus \{0\}$ and $ \beta >0$ are fixed constants. Then for each $\epsilon > 0$, $U_\epsilon$ satisfies
\[ (-\De)^su = |u|^{2^*_s-2}u \; \;\text{in} \; \mb R^n. \]
 Let $\nu >0$ {be} such that $B_{4\nu} \subset \Om$ and let $\zeta \in C^{\infty}_c(\mb R^n)$ be such that $0 \leq \zeta \leq 1$ in $\mb R^n$, $\zeta \equiv 0$ in {$\R^n\backslash B_{2\nu}$} and $\zeta \equiv 1$ in $B_\nu$. For each $\epsilon > 0$ and $x \in \mb R^n$, we define
$\Phi_\epsilon(x) :=  \zeta(x) U_\epsilon(x)$. From Lemma $4.12$ of \cite{TJS-ANA}, we have the following.
\begin{Lemma}\label{minimizer-gm}
$\sup \{{I_\la(u + t\Phi_\epsilon)}: t\geq 0\} <{ I_\la(u)} + {\frac{s}{n\la^{\frac{n-2s}{2s}}}}(S_s)^{\frac{n}{2s}}$, for {any} sufficiently small $\epsilon >0$.
\end{Lemma}

Now we prove the existence of second solution if $(MP)$ holds.
\begin{Lemma}\label{sec-sol-MP}
Let $(MP)$ holds then there exists a $v \in X_0$, distinct from $w$, which solves $(P_\la)$ weakly.
\end{Lemma}
\begin{proof}
From Lemma \ref{minimizer-gm}, it follows that there exists $\e>0$ and $R_0\geq 1$ such that
\begin{enumerate}
\item[(i)] $I_\la(w+RU_\e) < I_\la(w)$ for $\e \in (0,\e_0)$ and $R\geq R_0$.
\item[(ii)] $I_\la(w+tR_0 U_\e) < I_\la(w) + \displaystyle \frac{sS_s^{\frac{n}{2s}}}{n\la^{\frac{n-2s}{2s}}}$ for $\e \in (0,\e_0)$ and $t \in [0,1]$.
\end{enumerate}
We define the complete metric space
\[\Gamma := \{\eta \in C([0,1],T):\; \eta(0)=w,\; \|\eta(1)-w\|>\sigma_1,\; I_\la(\eta(1))< I_\la(w) \}\]
with metric defined as $d(\eta^\prime,\eta)=\max\limits_{t\in[0,1]}\{\|\eta^\prime(t)-\eta(t)\|\}$ for all $\eta,\eta^\prime \in \Gamma$. From $(i)$ above, we get that $\eta(t)= w+tR_0U_\e \in \Gamma$ for large enough $R_0>0$. This gives that $\Gamma \neq \emptyset$. Let $\gamma_0 = \inf\limits_{\eta\in \Gamma}\max\limits_{t \in [0,1]}I_\la(\eta(t))$ then by virtue of (ii) above and condition $(MP)$, we get
\[I_\la(w)< \gamma_0 \leq I_\la(w)+ \frac{sS_s^{\frac{n}{2s}}}{n}.\]
Now let $\Psi(\eta)= \max\limits_{t \in [0,1]}I_\la(\eta(t))$ for $\eta \in \Gamma$. Then using Ekeland's variational principle, we get a sequence $\{\eta_k\}\subset \Gamma$ such that
\begin{equation}\label{gm-12}
\Psi(\eta_k)<\gamma_0+ \frac{1}{k} \;\text{and}\;  \Psi(\eta_k){<} \Psi(\eta)+ \frac{1}{k}\|\Psi(\eta)-\eta(\eta_k)\|_\Gamma,\;\forall \eta\in \Gamma.
\end{equation}
We define
\[\Lambda_k = \{t \in [0,1]:\; I_\la(\eta_k(t)) = \max_{x\in [0,1]}I_\la(\eta_k(x))\}\].
\textbf{Claim:} There exists a $t_k \in \Lambda_k$  such that if $v_k = \eta_k(t_k)$ and $z \in T$ then
\[\int_{\mb R^n}(-\De)^sv_k(z-v_k) -\int_\Om (v_k^{-q}+\la v_k^{2^*_s-1})(z-v_k)~dx \geq -\frac{1}{k}\max\{1, \|z-v_k\|\}.\]
We prove it by contradiction, so assume that for every $t \in \Lambda_k$ there exists a $z_t\in T$ such that
\begin{equation}\label{gm-13}
\begin{split}
\int_{\mb R^n}(-\De)^s\eta_k(t)&\left( \frac{z_t-\eta_k(t)}{\max\{1,\|z_t-\eta_k(t)\|\}}\right)~dx\\
& - \int_{\Om}((\eta_k(t))^{-q}+\la (\eta_k(t))^{2^*_s-1})\left( \frac{z_t-\eta_k(t)}{\max\{1,\|z_t-\eta_k(t)\|\}}\right)~dx< -\frac{1}{k}.
\end{split}
\end{equation}
{Since $I_\la$ is locally Lipschitz} in $T$, $z_t$ can be chosen to be locally constant on $\Lambda_t$. Therefore for each $t \in \Lambda_k$ there exists a neighborhood $N_t$ of $t$ in $(0,1)$ such that for each $r \in N_t \cap \Gamma_k$, \eqref{gm-13} holds that is
\begin{equation}\label{gm-14}
\begin{split}
\int_{\mb R^n}(-\De)^s\eta_k(r)&\left( \frac{z_t-\eta_k(r)}{\max\{1,\|z_t-\eta_k(r)\|\}}\right)~dx\\
& - \int_{\Om}((\eta_k(r))^{-q}+\la (\eta_k(r))^{2^*_s-1})\left( \frac{z_t-\eta_k(r)}{\max\{1,\|z_t-\eta_k(r)\|\}}\right)~dx< -\frac{1}{k}.
\end{split}
\end{equation}
It is possible to choose a finite set $\{r_1,r_2,\ldots,r_m\} \subset \Lambda_k$ such that  $\Lambda_k \subset \cup_{i=1}^{m}J_{r_i}$. For notational convenience, we set $z_i=z_{r_i}$ and denote $\{\kappa_1,\kappa_2,\ldots,\kappa_m\}$ as the partition of unity associated with covering $\{J_{r_1}, J_{r_2}, \ldots,J_{r_m}\}$ of $\Lambda_k$. Now if we define $z(r)= \sum_{i=1}^{m}\kappa_i(r)z_i$ for ${r} \in [0,1]$  then $z(r) \in T$ for each ${r} \in [0,1]$. Therefore from \eqref{gm-14} we deduce that for all $r \in [0,1]$
\begin{equation}\label{gm-15}
\begin{split}
\int_{\mb R^n}(-\De)^s\eta_k(r)&\left( \frac{z(r)-\eta_k(r)}{\max\{1,\|z(r)-\eta_k(r)\|\}}\right)~dx\\
& - \int_{\Om}((\eta_k(r))^{-q}+\la (\eta_k(r))^{2^*_s-1})\left( \frac{z(r)-\eta_k(r)}{\max\{1,\|z(r)-\eta_k(r)\|\}}\right)~dx< -\frac{1}{k}.
\end{split}
\end{equation}
Let $h:[0,1] \to [0,1]$ be a continuous function such that $h(t)=1$ in a neighborhood of $\Lambda_k$ and $h(0)=h(1)=0$. Also we set $\mu_k(t)= \max\{1,\|z(t)-\eta_k(t)\|\}$ and
\[\eta(t)= \eta_k(t) +\frac{h(t)\epsilon}{\mu_k(t)}({z}(t)-\eta_k(t)).\]
Then for $\e\in (0,1)$, $\eta(t) \in T$ for all $t \in [0,1]$. Hence \eqref{gm-12} gives us that
\begin{equation}\label{gm-16}
\max_{t\in[0,1]}I_\la(\eta_k(t)) \leq \max_{t\in[0,1]}I_\la(\eta(t)) + \frac{\epsilon}{k}\max_{t\in[0,1]}\left(h(t)\frac{\|z(t)-\eta_k(t)\|}{\mu_k(t)}\right).
\end{equation}
If $t_{k,\e}\in[0,1]$ denotes the value such that $I_\la(\eta(t_{k,\e}))= \max_{t\in[0,1]}I_\la(\eta(t)) $ then we can assume that $t_{k,\e_j}\to t_k$ for some $t_k \in [0,1]$, where $\e_j$ is a sequence such that $\e_j \to 0$. Using the continuity of $\eta$, we deduce that
\[\eta({t_{k,\e_j}}) \to \eta_k(t_k)\;\text{as}\; \e_j\to 0.\]
Hence from \eqref{gm-16} we obtain that
$\max_{t\in[0,1]}I_\la(\eta_k(t)) \leq \max_{t\in[0,1]}I_\la(\eta_k(t_k))$ which implies $I_\la(\eta_k(t_k))= \max\limits_{t \in [0,1]} I_\la(\eta_k({t}))$. So $t_k \in \Gamma_k$ and $h(t_{k,\e_j})=1$ for $j>0$ large enough, by definition. If we set $v_k = \eta_k(t_k)$, $v_{k, j}= \eta_k(t_{k,\e_j})$ and $\mu_{k,j}= \max\{1,\|z(t_{k,\e_j})-v_{k,j}\|\}$ then for large enough $j$ we obtain
\begin{equation}\label{gm-17}
I_\la(v_{k,j}) \leq I_{\la}(v_k)\leq I_\la\left( v_{k,j}+ \frac{\e_j}{\mu_{k,j}}(z(t_{k,\e_j})- v_{k,j})\right) + \frac{\e_j}{k}.
\end{equation}
It is easy ro see that $\mu_{k,j} \to \theta_k:=  \max\{1,\|z(t_{k})-v_{k}\|\}$ and $\|v_k - v_{k,j}\| \to 0$ as $j \to \infty$. Let $p_j = v_{k,j}-v_k$ {and}
\[k_j = p_j + \e_j \left( \frac{z(t_{k,j})-v_{k,j}}{\mu_{k,j}} - \frac{z(t_k)-v_k}{\theta_k}\right) = p_j +o(1).\]
Then from \eqref{gm-17}, we obtain
\[\frac{1}{\e_j}\left(I_\la\left(v_k+\e_j\left(\frac{z(t_k)-v_k}{\theta_k} \right)+k_j \right)+ I_\la(v_k+p_j) \right) \geq -\frac{1}{k}\; \text{as}\; j \to \infty.\]
But since $v_k+\e_j\left(\frac{z(t_k)-v_k}{\theta_k} \right) \geq w$ using the fact that $z(t_k) \in T$, from Proposition \ref{gm-prop1} and the above inequality we get
\begin{equation*}\label{gm-15}
\begin{split}
\int_{\mb R^n}(-\De)^sv_k\left( \frac{z(t_k)-v_k}{\theta_k}\right)~dx- \int_{\Om}(v_k^{-q}+\la v_k^{2^*_s-1})\left( \frac{z(t_k)-v_k}{\theta_k}\right)~dx\geq -\frac{1}{k}.
\end{split}
\end{equation*}
This is a contradiction to {\eqref{gm-13}}. Thus, the claim holds. So there exists a sequence $\{v_k\}$ satisfying
\begin{equation}\label{gm-18}
\left\{
\begin{split}
&\int_{\mb R^n}(-\De)^sv_k(z-v_k) -\int_\Om (v_k^{-q}+\la v_k^{2^*_s-1})(z-v_k)~dx \geq -\frac{c}{k}(1+\|z\|)\; \text{for all}\; z \in T\\
&I_\la(v_k) \to \gamma_0\;\text{as}\; k \to \infty
\end{split}
\right.
\end{equation}
where $c>0$ is some constant. Setting $z = 2v_k$ in \eqref{gm-12} and using \eqref{gm-18} we get
\[\gamma_0+o(1) \geq \frac{ sC^n_s}{n}\|v_k\|^2 -\frac{2^*_s-1+q}{2^*_s(1-q)}\int_\Om |v_k|^{1-q}~dx- \frac{c}{2^*_s k}(1+2\|v_k\|).  \]
Now this implies that $\{v_k\}$ must be bounded in $X_0$, thus up to a subsequence, $v_k \rightharpoonup v$ weakly in $X_0$ as $k \to \infty$. Using similar ideas as in $(ZA)$ case, it can be shown that $v$ is a weak solution of $(P_\la)$.
Then the rest of the proof follows exactly same as Lemma $3.3$ of \cite{DPSS-cvee} or Lemma $2.7$ of \cite{haitao}. {\hfill {$\square$}\goodbreak \medskip}
\end{proof}

\noi \textbf{Proof of Theorem \ref{maintheorem-gm}:}
The proof follows from Lemma \ref{sec-sol-ZA}, Lemma \ref{sec-sol-MP} and Proposition \ref{firstsol-gm} along with Proposition \ref{gm-prop2}. {\hfill {$\square$}\goodbreak \medskip}

\noi \textbf{Proof of Theorem \ref{gm-reg2}:}
The proof follows directly from Proposition \ref{reg-gm1} (Appendix) and Theorem $1.2$ of \cite{AJS}. {To see that the regularity result falls into the scope of Theorem $1.2$ of \cite{AJS}, note that $\underline{u_\lambda}\leq u\leq z_\lambda$ (refer to Appendix). Moreover, from the fact that $\underline{u_\lambda}$ and $z_\lambda$ are in $C^+_{\phi_q}(\Omega)$ together with local regularity results from \cite{Silvestre-CPAM}, we infer that $u\in C^+_{\phi_q}(\Omega)$.}

{\hfill {$\square$}\goodbreak \medskip}

\section{Appendix}
In this section, we prove that any weak solution of $(P_\la)$ is in $L^\infty(\Om)$. We prove it in the spirit of Proposition $2.2$ of \cite{serv-Linfty}.
First we let $u\in X_0$ denotes any weak solution of $(P_\la)$ and we know that $\underline{u_\la}\in X_0 \cap C_{\phi_q}^+(\Om)$ (defined in Lemma \ref{gm-lem1}) forms a subsolution of $(P_\la)$ satisfying
$(-\De)^s\uline{u_\la} = \uline{u_\la}^{-q}$ in $\Om$.\\
\textbf{Claim :} $\uline{u_\la}\leq u$ a.e. in $\Om$.\\
Suppose it is not true. Then it is easy to see that for any $v\in X_0$ it holds
\[(v(x)-v(y))(v^+(x)-v^-(y)) \geq |v^+(x)-v^+(y)|^2,\; \text{for any}\; x,y\in R^n.\]
Therefore using $(\uline{u_\la}-u)^+$ as the test function in
\[(-\De)^s(\uline{u_\la}-u) \leq \uline{u_\la}^{-q}- u^{-q}\; \text{in}\; \Om\]
we get
\begin{align*}
0&\leq C^n_s\int_{Q}\frac{|(\uline{u_\la}-u)^+(x)- (\uline{u_\la}-u)^+(y)|^2}{|x-y|^{n+2s}}~dxdy\\
 &\leq C^n_s\int_{Q}\frac{((\uline{u_\la}-u)^+(x)- (\uline{u_\la}-u)^+(y))((\uline{u_\la}-u)(x)- (\uline{u_\la}-u)(y))}{|x-y|^{n+2s}}~dxdy\\
&\leq \int_\Om (\uline{u_\la}^{-q}- u^{-q})(\uline{u_\la}-u)^+~dx\leq 0.
\end{align*}
Hence it must be that meas$\{x\in \Om:\; \uline{u_\la}(x)\geq u(x)\}=0$ which establishes our claim. {Also if  $z_\lambda$ is defined as the unique solution (refer Theorem $1.1$ of \cite{AJS}) to
\begin{equation*}
(-\De)^s z_\lambda =  {z_\lambda}^{-q} + \la c, \quad u>0 \; \text{in}\;
\Om,\quad u = 0 \; \mbox{in}\; \mb R^n \setminus\Om,
\end{equation*}
with $c=\|u\|_{\infty}^{2^*_s-1}$ then similarly we can prove that $u\leq z_\la$.}

\begin{Proposition}\label{reg-gm1}
If $u\in X_0 $ is any  weak solution of $(P_\la)$ for $\la \in (0,\Lambda]$ then $u \in L^\infty(\Om)$.
\end{Proposition}
\begin{proof}
Let $u\in X_0$ denotes a  weak solution of $(P_\la)$. Then by virtue of the above claim and Hardy's inequality, we know that $\int_\Om u^{-q}\phi~dx<\infty$ for any $\phi\in X_0$. We aim to show that $(u-1)^+$ belongs to $L^\infty(\Om)$ which will imply that $u \in L^\infty(\Om)$. If $f(t)=(t-1)^+$ for $t\in \mb R$ and $\psi(t) \in C^\infty(\mb R)$ be a convex and increasing function such that $\psi^\prime(t)\leq 1$ when $t \in[0,1]$ and $\psi^\prime(t)=1$ when $t\geq 1$ then we can define
\[\psi_\e(t) = \e \psi(t/\e)\]
so that $\psi_\e \to f$ uniformly as $\e\to 0$. Also since $\psi_\e$'s are smooth, by regularity results and the uniform convergence of $\psi_\e$ to $f$ we get that
\[(-\De)^s \psi_\e(u) \to (-\De)^s (u-1)^+ \;\text{as}\; \e \to 0.\]
Moreover because $\psi_\e$'s are convex and differentiable, we know that
\[(-\De)^s \psi_\e (u) \leq \psi^\prime_\e(u)(-\De)^s u \leq \chi_{\{u>1\}} (-\De)^s u\]
where $\chi_{\{u>1\}}$ denotes the characteristic function over the set $\{x\in \Om :\; u(x)>1\}$. Then passing on the limits $\e \to 0$ in above equation, we obtain
\[(-\De)^s(u-1)^+ \leq \chi_{\{u>1\}} (-\De)^s u \leq \chi_{\{u>1\}}(u^{-q}+ \la u^{2^*_s-1}) \leq C(1+ ((u-1)^+)^{2^*_s-1})\]
for some constant $C>0$. Therefore we use Proposition $2.2$ of \cite{serv-Linfty} to conclude that $(u-1)^+ \in L^\infty(\Om)$. This completes the proof.{\hfill {$\square$}\goodbreak \medskip}
\end{proof}


\begin{thebibliography}{21}
\linespread{0.1}

\bibitem{AJS} Adimurthi, J. Giacomoni and S. Santra, {\it Positive solutions to a fractional equation with singular nonlinearity}, J. Differential Equations, 265 (4) (2018), 1191-1226.

\bibitem{da}D. Applebaum, {\it L$\acute{e}$vy process-from probability to finance and quantum groups}, Notices Amer. Math. Soc.,
51 (2004) 1336-1347.

\bibitem{serv-Linfty} B. Barrios, E. Colorado, R. Servadei and F. Soria, {\it A critical fractional equation with concave�convex power nonlinearities}, Ann. I. H. Poincare, 32 (2015) 875-900.
\bibitem{peral}  B. Barrios, I. De Bonis,  M. Medina and  I. Peral, {\it  Semilinear problems for the fractional laplacian with a singular nonlinearity}, Open Math., 13 (2015) 390--407.

\bibitem{peral-hardy} B. Barrios, M. Medina and I. Peral, {\it Some remarks on the solvability of non-local elliptic problems with the Hardy potential}, Communications in Contemporary Mathematics, 16 (4) (2014), 1350046 (29 pages).


\bibitem{buccur} C. Bucur and E. Valdinoci, {\it An Introduction to the Fractional Laplacian}, Nonlocal Diffusion and Applications, Lecture Notes of the Unione Matematica Italiana, 20 (2016), Springer, Cham.

\bibitem{cai-chu} Z. Cai, C. Chu and C. Lei, {\it Existence of positive solutions for a
fractional elliptic problems with the
Hardy-Sobolev-Maz�ya potential and critical
nonlinearities}, Boundary Value Problems, (2017), DOI 10.1186/s13661-017-0912-8.

\bibitem{chen} W. Chen, {\it Fractional elliptic problems with two critical
sobolev-hardy exponents}, Electronic Journal of Differential Equations, 22 (2018), 1�12.

\bibitem{DPSS-cvee} R. Dhanya, S. Prashanth, Sweta Tiwari and K. Sreenadh, {\it Elliptic
Problems in with Critical and Singular Discontinuous Nonlinearities}, Complex Variables and Elliptic Equations, 61(12) (2016), 1656-1676.


\bibitem{FP} A. Fiscella and P. Pucci, {\it On certain nonlocal Hardy-Sobolev critical elliptic Dirichlet problems},
Adv. Differential Equations, 21 (5/6) (2016), 571-599.

\bibitem{TJS-ANA} J. Giacomoni, T. Mukherjee and K. Sreenadh, {\it Positive solutions of fractional elliptic
equation with critical and singular nonlinearity}, Advances in Nonlinear Analysis, 6(3) (2017), 327�354.

\bibitem{TJS-para} J. Giacomoni, T. Mukherjee and K. Sreenadh, {\it Existence and stabilization results for a singular parabolic equation involving the fractional Laplacian}, to appear in Discrete and Continuous Dynamical Systems - Series S.

\bibitem{haitao} Y. Haitao, {\it Multiplicity and asymptotic behavior of positive solutions for a singular semilinear elliptic problem}, J. Differential Equations, 189 (2003) 487--512.

\bibitem{BRS} G. Molica Bisci, V. Radulescu and R. Servadei, {\it Variational Methods for Nonlocal Fractional Problems (Encyclopedia of Mathematics and its Applications)}, Cambridge University Press, Cambridge, DOI:10.1017/CBO9781316282397.

\bibitem{TS-ejde} T. Mukherjee and K. Sreenadh, {\it Critical growth fractional elliptic problem with singular nonlinearities}, Electronic Journal of differential equations, 54 (2016) 1--23.

{\bibitem{Ros-oton-serra-JMPA} X. Ros-Oton and J. Serra, {\it The Dirichlet problem for the fractional Laplacian: Regularity up to the boundary}, J. Math. Pures Appl., 101 (2014), 275--302.}

\bibitem{sv} R. Servadei and E. Valdinoci, {\it The Brezis-Nirenberg result for the fractional laplacian}, Trans. Amer. Math. Soc., 367 (2015) 67--102.
{\bibitem{Silvestre-CPAM} L. Silvestre, {\it Regularity of the obstacle problem for a fractional power of the Laplace operator}, Comm. Pure Appl. Math., 60 (1) (2007), 67--112.}

 \end{thebibliography}
\end{document}